\documentclass[12pt,reqno]{amsart}
\usepackage{amsmath,amsthm,amsfonts,amssymb}
\usepackage{tikz}
\usepackage{graphicx}
\usepackage{psfrag}
\usepackage{mathrsfs}
\numberwithin{equation}{section}

\newcommand{\rank}{\operatorname{rank}}
\newcommand{\image}{\operatorname{im}}

\newcommand{\dem}{\begin{proof}}
\newcommand{\cqd}{\end{proof}}

\def\to{\mathop{\rightarrow}}

\newtheorem{theorem}{Theorem}[section]

\newtheorem{proposition}[theorem]{Proposition}

\newtheorem{claim}{Claim}

\newtheorem{lemma}[theorem]{Lemma}

\newtheorem{definition}[theorem]{Definition}

\newtheorem{obs}[theorem]{Remark}

\newtheorem{example}[theorem]{Example}

\newtheorem{question}[theorem]{Question}

\begin{document}


\subjclass[2000]{}

\date{\today}

\title{Tori with hyperbolic dynamics in 3-manifolds}

\begin{abstract}
Let $M$ be a closed orientable irreducible
3-dimensional manifold, and let $f:M\to M$ be a diffeomorphism. We call an embedded 2-torus ${\mathbb T}$ {\em Anosov torus} if ${\mathbb T}$ is $f$-invariant and $f_\#|_\mathbb{T}:\pi_1(\mathbb{T})\rightarrow \pi_1(\mathbb{T})$ is
hyperbolic. We prove that only few irreducible 3-manifolds admit Anosov tori: (1) the 3-torus ${\mathbb T}^3$, (2) the mapping torus of $-id$, and (3) the mapping torus of hyperbolic automorphisms of ${\mathbb T}^2$ .\par
This has consequences for instance in the context of partially hyperbolic dynamics of 3-manifolds: if there is an invariant foliation ${\mathcal F}^{cu}$ tangent to the center-unstable bundle $E^c\oplus E^u$, then $\mathcal{F}^{cu}$ has no compact leaves \cite{rru_nocoherent}. This has led to the first example of a non-dynamically coherent partially hyperbolic diffeomorphism with one-dimensional center bundle. \cite{rru_nocoherent}.
\end{abstract}
\thanks{}

\author{F. Rodriguez Hertz}
\address{IMERL-Facultad de Ingenier\'\i a\\ Universidad de la
Rep\'ublica\\ CC 30 Montevideo, Uruguay.}
\email{frhertz@fing.edu.uy}\urladdr{http://www.fing.edu.uy/$\sim$frhertz}

\author{M. A. Rodriguez Hertz}
\address{IMERL-Facultad de Ingenier\'\i a\\ Universidad de la
Rep\'ublica\\ CC 30 Montevideo, Uruguay.}
\email{jana@fing.edu.uy}\urladdr{http://www.fing.edu.uy/$\sim$jana}

\author{R. Ures}
\address{IMERL-Facultad de Ingenier\'\i a\\ Universidad de la
Rep\'ublica\\ CC 30 Montevideo, Uruguay.} \email{ures@fing.edu.uy}
\urladdr{http://www.fing.edu.uy/$\sim$ures}

\maketitle

\section{Introduction}
In this article we address the issue of whether an embedded torus
in a 3-manifold is left invariant by a diffeomorphism
inducing a hyperbolic automorphism in its first fundamental group.
Our motivation in studying these objects come from many interesting problems in partially hyperbolic dynamics, which will be discussed shortly afterwards.  However, let us remark that this result is not entirely devoid of interest in the topological classification of 3-manifolds.
\par
Indeed, there is a unique minimal way, up to isotopy, of cutting an irreducible oriented 3-manifold along embedded disjoint incompressible tori, into pieces that are either Seifert or atoroidal manifolds. This decomposition is called {\em JSJ decomposition}, and is unique up to isotopy.
The way in which these Seifert and atoroidal manifold are glued together along the incompressible tori determines the type of the 3-manifold. In this way, the same family of basic (Seifert or atoroidal) components can produce non-diffeomorphic 3-manifolds if the gluing diffeomorphisms are different.  Here we give some step toward the classification of 3-manifolds. As an example of how our result maybe used in this direction we have that for instance, if $M$ and $M'$ are build up by the same two components, there is only one cutting torus in the JSJ-decomposition, the gluing diffeomorphism in $M$ is the identity, and in $M'$ is hyperbolic, then $M$ and $M'$ are not diffeomorphic. Let us note that if two 3-manifolds are diffeomorphic, the cutting tori in the JSJ-decomposition can be taken to be invariant under this diffeomorphism (see Remark \ref{remark f deja invariante JSJ}). Our main result (Theorem \ref{theorem principal}) also shows that, except in some very particular cases, the diffeomorphism will not be isotopic to hyperbolic on any cutting torus. \par
We shall say that a 2-torus ${\mathbb T}$ embedded in a 3-manifold $M$ is an {\em Anosov torus} if there exists a diffeomorphism $f$ over $M$ such that the induced action of $f$ over the fundamental group of ${\mathbb T}$ is hyperbolic. Our main result is the following.\par
\begin{theorem}\label{theorem principal}
A closed oriented irreducible 3-manifold admits an Anosov torus if and only if it is one of the following:
\begin{enumerate}
\item the 3-torus
\item the mapping torus of $-id$
\item the mapping torus of a hyperbolic automorphism
\end{enumerate}
\end{theorem}
Moreover, we have the following result:
\begin{theorem}\label{theorem principal con borde}
Let $M$ be a compact orientable irreducible 3-manifold with non-empty boundary such that all the boundary components are incompressible 2-tori. Then $M$ admits an Anosov torus if and only if $M={\mathbb T}^2\times[0,1]$.
\end{theorem}

One main reason to consider only irreducible 3-manifolds is that this work arises in the context of partially hyperbolic systems (see below), and only irreducible 3-manifolds admit such dynamics \cite{burago_ivanov}. But on the other hand, it is important to note that it is easy to construct arbitrarily many different non-irreducible manifolds admitting Anosov tori. Indeed, if a manifold $M$ supports an Anosov torus, then it is easy to see that the connected sum of $M$ with any other 3-manifold will admit an Anosov torus. See Remark \ref{remark irreducible manifolds}.
An interesting question to solve would be the following:
\begin{question}
Let $M$ be an orientable non-irreducible 3-manifold. Then, the Kneser-Milnor theorem states that we can decompose $M$, uniquely up to diffeomorphisms, into a finite connected sum:
$$M=M_1\#M_2\#\dots\# M_n$$
where each $M_i$ is either irreducible, or a handle ${\mathbb S}^2\times {\mathbb S}^1$.
If $M$ admits an Anosov torus, is one of the $M_i$ necessarily one of the 3-manifolds listed in Theorem \ref{theorem principal}?
\end{question}
The technical difficulties in answering this question mainly arise from the non-uniqueness of the prime decomposition modulo isotopies. \newline\par

The study of Anosov tori is relevant in the context of {\em partially hyperbolic dynamics} of 3-manifolds. A partially hyperbolic system is a diffeomorphism that leaves invariant three complementary bundles: $E^s$, on which the action of the derivative is contracting; $E^u$, on which it is expanding; and $E^c$ on which the action is not as contracting as in $E^s$, nor as expanding as $E^u$. See for instance \cite{rru_survey}.\par%
It was conjectured by Pugh and Shub in 1995 that conservative partially hyperbolic diffeomorphisms contain an open and dense set of ergodic systems. This conjecture was proven true for 3-manifolds by the authors \cite{rru_invent}. Our question is: can we classify all 3-manifolds supporting non-ergodic partially hyperbolic diffeomorphisms? We conjecture the following:
\begin{question}[Conjecture]
A 3-manifold supports a non-ergodic partially hyperbolic diffeomorphism if and only if it is one of the following:
\begin{enumerate}
\item the 3-torus
\item the mapping torus of $-id$, or
\item the mapping torus of a hyperbolic automorphism of the 2-torus.
\end{enumerate}
\end{question}
In \cite{rru_nilman} we proved that this conjecture is true for nilmanifolds: the only nilmanifold supporting a non-ergodic conservative partially hyperbolic diffeomorphism is the 3-torus.
Note that, according to this conjecture, the only manifold supporting non-ergodic partially hyperbolic diffeomorphisms
would be exactly those enumerated in Theorem \ref{theorem principal}, that is the manifolds admitting an Anosov torus.
Let us remark that 3-manifolds that support partially hyperbolic diffeomorphisms, ergodic or not, are always {\em irreducible}, that is every 2-sphere embedded in the manifold bounds a 3-ball \cite{burago_ivanov}.\newline\par
Another problem in partially hyperbolic dynamics concerns the integrability of $E^c$, the {\em center bundle}, that is, of the bundle whose expansion and contraction rates are bounded, respectively, by the contraction and expansion of the derivative on the {\em strong bundles} $E^s$ and $E^u$. Indeed, the strong bundles are known to be integrable \cite{hirsch_pugh_shub}, but the situation is different with the center bundle. It is an open problem to determine the conditions under which a partially hyperbolic diffeomorphism of a 3-manifold has an integrable center bundle. In \cite{burago_ivanov}, it is shown that there are always foliations {\em almost} tangent to the center bundle. Moreover, a large class of partially hyperbolic diffeomorphisms of the 3-tori have integrable center bundle, as it was recently shown in \cite{brin_burago_ivanov2009}. However, in \cite{rru_nocoherent} we give an example of a partially hyperbolic diffeomorphism of the 3-torus, having a non-integrable center bundle. This example answers a question that had been posed by many authors in the last decades, basically since partially hyperbolic diffeomorphisms where defined, see for instance \cite{hirsch_pugh_shub} and \cite{brin_pesin}.\par
The example we obtained was inspired in the theorem below, which is one of the applications of Theorem \ref{theorem principal} and gives a more accurate description of some dynamically defined foliations of partially hyperbolic diffeomorphisms of 3-manifolds.

\begin{theorem}[\cite{rru_nocoherent}]\label{uniqueint} Let $M$ be a closed orientable 3-dimensional manifold and $f:M\rightarrow M$ be a partially hyperbolic diffeomorphism with dynamically coherent center-unstable bundle $E^c\oplus E^u$. Then, the center-unstable foliation ${\mathcal F}^{cu}$ has no compact leaves.
\end{theorem}
Indeed, to study the integrability of the center bundle, it is standard to analyze the behavior of the so called {\em center-stable} and {\em center-unstable} bundles, that are respectively, the Whitney sums $E^s\oplus E^c$ of the center and contracting ones, and $E^c\oplus E^u$ of the center and expanding ones. One says that the center-stable bundle $E^s\oplus E^c$ is {\em dynamically coherent} if it is integrable, that is, if there exists an invariant foliation tangent to it. In this case, the tangent foliation is called {\em center-stable foliation}. Analogously one defines the center-unstable foliation. \par
Observe that Theorem \ref{uniqueint} does not prevent the existence of tori, even invariant, tangent to the center-unstable bundle. Theorem \ref{uniqueint} asserts the impossibility of the existence of such tori as part of an invariant foliation tangent to the center-unstable bundle. In \cite{rru_nocoherent} we give examples of a partially hyperbolic diffeomorphism of $\mathbb{T}^3$ with  center unstable tori. The center foliation of this example in particular is not uniquely integrable although some of these examples are {\em dynamically coherent}, that is both the center-stable and the center-unstable bundles are dynamically coherent.\newline\par

We can mention a third remarkable problem in partially hyperbolic dynamics, that is their classification in 3-manifolds. In particular, the classification of the manifolds that support such a dynamics. In this context, let us note that any compact manifold tangent to the strong bundles is an Anosov torus. Theorem \ref{theorem principal} precludes the existence of such tori except in the above mentioned cases. \newline\par
The idea of the proof is the following: first, let us say that the sufficiency is straightforward, and is explained in Section \ref{s_teo1}. The proof of the necessity, on the other hand, is strongly based in the so-called JSJ-decomposition (Theorem \ref{teorema jsj}): A compact irreducible orientable 3-manifold can be decomposed, uniquely up to isotopy, by cutting along incompressible tori, into components that are either Seifert manifolds, or atoroidal manifolds. A Seifert manifold is a manifold foliated by circles, and an atoroidal manifold is one that does not admit incompressible tori except, possibly, those isotopic to a component of its boundary. Now, an Anosov torus is always incompressible (Theorem \ref{teorema anosov incompresible}), and this implies that the JSJ-decomposition can be chosen so that no atoroidal component contains it. In Proposition \ref{proposition reduccion} we prove that there are three possibilities for an Anosov torus: either (a) it is one of the cutting tori, or else it is contained in one of the Seifert components, in which case it is either (b) union of circles of the Seifert foliation or (c) transverse to all circles of the Seifert foliation. We explain below how these three cases are dealt with.\par
Let us mention that some cases in Theorem \ref{theorem principal} had been already studied, although not with this name, by F. Waldhausen in his classification of graph manifolds \cite{waldhausen}; namely, the case where all the components in the JSJ-decomposition are Seifert manifolds. For the sake of completeness we also present proofs of these cases.\newline\par%
The paper is organized as follows. We will include preliminary concepts of 3-manifolds in order that this paper be as self-contained as possible, see Section \ref{preliminaries}. The goal of Section \ref{preliminaries} is to introduce all concepts that are necessary to understand Proposition \ref{proposition reduccion}, which contains the architecture of the proof of Theorem \ref{theorem principal}, as mentioned in the previous paragraphs, the devoted reader may go to the book \cite{hatcher} for complete accounts on the topology of 3-manifolds. . Section \ref{section proof proposition} is devoted to proving this proposition. Next sections are devoted to proving cases (\ref{item toro en JSJ decomposition}), (\ref{item vertical torus}) and (\ref{item horizontal torus}) mentioned above.\par
In Section \ref{seccion horizontal torus} we prove case (\ref{item horizontal torus}), namely, the case in which the manifold $M$ is Seifert and the Anosov torus is transverse to all the fibers of the Seifert fibration, that is, it is a {\em horizontal} torus. This is reduced to a case by case proof, since there are only six Seifert manifolds that admit a horizontal torus. We prove in this case that $M$ is either the 3-torus, or the mapping torus of $-id$.\par
In Section \ref{section vertical torus} we study the case (\ref{item vertical torus}), that is, the case in which the Anosov torus is contained in a Seifert component and it is union of Seifert fibers, hence it is a {\em vertical} torus.  We prove also in this case that $M$ is either the 3-torus, or the mapping torus of $-id$. There is also the possibility that the Seifert component be ${\mathbb T}^2\times[-1,1]$, but this possibility is studied in Section \ref{section atoroidal}.\par
In Section \ref{section atoroidal} we study the last case (\ref{item toro en JSJ decomposition}), in which the Anosov torus is part of the JSJ-decomposition. We prove in this case that the the Anosov torus cannot be boundary of an atoroidal component, unless the component is ${\mathbb T}^2\times[-1,1]$. \par
Finally, in Section \ref{s_teo1}, we finish the proofs of Theorems \ref{theorem principal} and \ref{theorem principal con borde}. The idea, as follows from Sections \ref{seccion horizontal torus}, \ref{section vertical torus} and \ref{section atoroidal}, is that either $M$ is Seifert and it is the 3-torus or the mapping torus of $-id$, or else, the JSJ-decomposition consists of only one torus (the Anosov torus), and in this last case, after cutting $M$ along this torus, we obtain the Anosov torus cross the interval. To reobtain $M$, we have to glue the boundary tori by means of a diffeomorphism which will have to commute with the hyperbolic dynamics on the Anosov torus. The only possibilities for the restriction of the gluing diffeomorphism to the Anosov torus are: $\pm id$ and a hyperbolic diffeomorphism. In the first case, we have $M$ is the mapping torus of $\pm id$, hence it is Seifert; and in the last case, we obtain that $M$ is the mapping torus of a hyperbolic automorphism of the 2-torus. \par
The rest of this section is devoted to proving the sufficiency part of Theorem \ref{theorem principal}; and Theorem \ref{theorem principal con borde}, which follows from Theorem \ref{theorem principal}.
\section{Preliminaries}\label{preliminaries}
Let $M$ be a 3-dimensional manifold. In this work we classify irreducible 3-manifolds admitting Anosov tori. A manifold $M$ is {\em irreducible} if every 2-sphere $\mathbb{S}^2$ embedded in the manifold bounds a 3-ball. A 2-torus $T$ embedded in $M$ is an {\em Anosov torus} if there exists a diffeomorphism $f:M\to M$ such that
$f(T)=T$ and the action induced by $f$ on $\pi_1(T)$, that is, $f_\#|_{T}:\pi_1(T)\to\pi_1(T)$, is a hyperbolic automorphism. Equivalently, $f$ restricted to $T$ is isotopic to a hyperbolic automorphism. \par
We shall assume from now on, that $M$ is an irreducible 3-manifold. In what follows, we will focus on what is called the JSJ-decomposition of $M$ (see below). That is, we will cut $M$ along certain kind of tori, called incompressible, and will obtain certain 3-manifolds with boundary that are easier to handle, which are, respectively, Seifert manifolds, and atoroidal and acylindrical manifolds. Let us introduce these definitions first.\par
An orientable surface $S$ embedded in $M$ is {\em incompressible} if
the homomorphism induced by the inclusion map
$i_\#:\pi_1(S)\hookrightarrow\pi_1(M)$ is injective; or,
equivalently, if there is no embedded disc
$D^2\subset M$ such that $D\cap S=\partial D$ and $\partial D\nsim
0$ in $S$ (see, for instance, \cite[Page 10]{hatcher}). We also require that $S\ne{\mathbb S}^2$.\par
A manifold with or without boundary is {\em Seifert}, if it admits a one dimensional foliation by closed curves, called a Seifert fibration. The boundary of a Seifert manifold with boundary consists of finite union of tori. There are many examples of Seifert manifolds, for instance ${\mathbb S}^3$. See also Example \ref{model seifert fibering} for models of Seifert manifolds. \par
The other type of manifold obtained in the JSJ-decomposition is atoroidal and acylindrical manifolds. A 3-manifold with boundary $N$ is {\em atoroidal} if every incompressible torus is {\em $\partial$-parallel}, that is, isotopic to a subsurface of $\partial N$. A 3-manifold with boundary $N$ is {\em acylindrical} if every incompressible annulus $A$ that is {\em properly embedded}, i.e.  $\partial A\subset \partial N$, is $\partial$-parallel, by an isotopy fixing $\partial A$.\par
As we mentioned before, a closed irreducible 3-manifold admits a natural decomposition into Seifert pieces on one side, and atoroidal and acylindrical components on the other:
\begin{theorem}[JSJ-decomposition \cite{jaco_shalen}, \cite{johannson}]\label{teorema jsj}
If $M$ is an irreducible closed orientable 3-manifold, then there exists a collection of disjoint incompressible tori ${\mathcal T}$ such that each component of $M\setminus {\mathcal T}$ is either Seifert, or atoroidal and acylindrical. Any minimal such collection is unique up to isotopy. This means, if ${\mathcal T}$ is a collection as described above, it contains a minimal sub-collection $m({\mathcal T})$ satisfying the same claim. All collections $m({\mathcal T})$ are isotopic.\end{theorem}
Any minimal family of incompressible tori as described above is called a {\em JSJ-decomposition} of $M$. When it is clear from the context we shall also call JSJ-decomposition the set of pieces obtained by cutting the manifold along these tori. Note that if $M$ is either atoroidal or Seifert, then ${\mathcal T}=\emptyset$. \par
The idea of the proof of Theorem \ref{theorem principal} is that, given an Anosov torus $T$, we can ``place" $T$ so that either $T$ belongs to the family ${\mathcal T}$, or else $T$ is in a Seifert component, and it is either transverse to all fibers, or it is union of fibers of this Seifert component. See Proposition \ref{proposition reduccion}.\par
It is important to note the following property of Anosov tori:
\begin{theorem}\cite{rru_nilman}\label{teorema anosov incompresible}
Anosov tori are incompressible.
\end{theorem}
An Anosov torus in an atoroidal component will then be $\partial$-parallel to a component of its boundary. In this case, we can assume $T\in {\mathcal T}$. On the other hand, the Theorem of Waldhausen below, guarantees that we can always place an incompressible torus in a Seifert manifold in a ``standard" form; namely, the following:  a surface is {\em horizontal} in a Seifert manifold if it is transverse to all fibers, and {\em vertical} if it is union of fibers:
\begin{theorem}[Waldhausen \cite{waldhausen}]\label{teorema waldhausen} Let $M$ be a compact connected Seifert manifold, with or without boundary. Then any incompressible surface can be isotoped to be horizontal or  vertical.
\end{theorem}
The architecture of the proof of Theorem \ref{theorem principal}, as mentioned above, is contained in the following proposition.
\vspace*{1em}
\begin{proposition}\label{proposition reduccion} Let $T$ be an Anosov torus of a closed irreducible orientable manifold $M$. Then, there exists a diffeomorphism $f:M\rightarrow M$ and a JSJ-decomposition ${\mathcal T}$ such that
\begin{enumerate}
\item \label{item hyperbolic toral automorphism} $f|T$ is a hyperbolic toral automorphism,
\item \label{item f deja invariante JSJ} $f({\mathcal T})={\mathcal T}$, and
\item \label{item JSJ-decomposition}one of the following holds
\begin{enumerate}
\item\label{item toro en JSJ decomposition} $T\in{\mathcal T}$
\item\label{item vertical torus} $T$ is a vertical torus in a Seifert component of $M\setminus{\mathcal T}$, and $T$ is not $\partial$-parallel in this component.
\item\label{item horizontal torus} $M$ is a Seifert manifold (${\mathcal T}=\emptyset$), and $T$ is a horizontal torus,
\end{enumerate}
\end{enumerate}
\end{proposition}
\vspace*{1em}
The proposition above allows us to split the proof of Theorem \ref{theorem principal} into cases.
Note that case (\ref{item vertical torus}) includes the case in which $M$ is a Seifert manifold and $T$ is a vertical torus.
Before addressing to the proof of Proposition \ref{proposition reduccion}, which is done in Section \ref{section proof proposition}, let us describe better the models of Seifert manifolds.
\begin{example}\label{model seifert fibering} {\em The {\em model Seifert fibering} of the solid torus ${\mathbb D}^2\times{\mathbb S}^1$, or {\em standard fibered torus} consists of the orbits of the suspension of the rotation of $D^2$ by the angle $2\pi p/q$, where $p$ and $q$ are coprime integers. If $q\ne 1$, then all the fibers $\{z\}\times {\mathbb S}^1$ with $z\ne0$ are {\em regular}: they have a neighborhood where the fibration is homeomorphic to a product fibration. $\{0\}\times {\mathbb S}^1$ is {\em exceptional}, that is, not regular.}
\end{example}
Model Seifert fiberings are fundamental in the geometric description of Seifert manifolds, due to the following:
\begin{theorem}[Epstein \cite{epstein}] \label{teorema epstein} Every fiber in an arbitrary Seifert manifold has a neighborhood which is fiber-preserving diffeomorphic to a neighborhood of a fiber in a model Seifert fibering.
\end{theorem}
In fact, this was the original definition of Seifert fibering. \newline\par
The Theorem of Epstein has a consequence which shall be useful in proving Proposition \ref{proposition reduccion}, and is interesting in itself:
\begin{lemma}\label{lema horizontal surface} Let ${\mathcal S}$ be a Seifert fibering of a compact orientable irreducible 3-manifold $M$. Then:
\begin{enumerate}
\item If a surface with or without boundary is horizontal, then it intersects all the fibers of ${\mathcal S}$.
\item If $\partial M\ne\emptyset$ then $\mathcal S$ does not admit horizontal surfaces without boundary.
\end{enumerate}
\end{lemma}
\begin{proof}
Indeed, let $A$ denote the set of points, the fiber of which has a non empty transverse intersection with the horizontal surface $S$. Then $A$ is clearly open. To see that $A$ is closed take a sequence $x_n\to x$ such that $x_n\in A$. If the fiber of $x$ did not intersect the horizontal surface, then there would be a neighborhood of the fiber of $x$ not intersecting the horizontal surface. Theorem \ref{teorema epstein} implies that in fact there is a fibered neighborhood of the fiber of $x$ not intersecting $S$. This neighborhood would contain the fibers of $x_n$, an absurd. \par
To see that horizontal manifold without boundary can live only in manifolds without boundary, assume that $T$ is a horizontal surface without boundary and assume $\partial M$ is not empty. Consider a fiber of $x\in\partial M$. This fiber intersects $T$. But then, due Theorem \ref{teorema epstein}, there is a fibered neighborhood of the fiber of $x$ diffeomorphic to a solid torus. This contradicts that $x$ is in the boundary of $M$.
\end{proof}
Let us finish the section by explaining a little bit why we focused in irreducible manifolds.
\begin{obs}\label{remark irreducible manifolds} It is easy to see that there are arbitrarily many non-irreducible 3-manifolds admitting Anosov tori. Indeed, let $M$ be any closed 3-manifold admitting an Anosov torus $T$. We loose no generality in assuming that $f:M\to M$ is such that $f|_T$ has a fixed point $p$. It is easy to see that we can slightly modify $f$ so that, if $T\times [-\epsilon,\epsilon]$ is a small tubular neighborhood of $T$, then $f|_{T\times\{t\}}=f|_{T\times\{0\}}=f|_T$ for all $t\in[-\epsilon,\epsilon]$. \par
Let us make another slight modification of $f$: replace $p\times\{\epsilon\}\subset T\times\{\epsilon\}$ by a small ball $B\subset M$ and take $g:M\to M$ so that $g=f$ on $M\setminus \{p\times\{\epsilon\}\}$, and $g$ restricted to $B$ is the identity.\par
If we consider now any manifold $M'$, then there is a diffeomorphism $h$ on $M\#M'$ such that $h$ is $f$ when restricted to $M\setminus B$ and $h$ is the identity when restricted to $M'$ minus another small ball. This implies that $M\#M'$ admits an Anosov torus. In this way we can construct arbitrarily many manifolds admitting Anosov tori.
\end{obs}
%
%
\section{Proof of Proposition \ref{proposition reduccion}}\label{section proof proposition}
This subsection contains the proof of Proposition \ref{proposition reduccion}.
Let $M$ be an irreducible orientable closed 3-manifold and $T$ be an Anosov torus.
Firstly, note that we can choose a JSJ-decomposition ${\mathcal T}$  such that $T$ is in a Seifert piece. This is because of the so-called Enclosing Property:
\begin{proposition}[Enclosing Property \cite{johannson}] \label{proposition enclosing property}
There exists ${\mathcal T}$ such that either $T\in{\mathcal T}$, or else $T$ is contained in the interior of a Seifert piece of the JSJ-decomposition generated by ${\mathcal T}$, and is not $\partial$-parallel in that component.
\end{proposition}
Hence, either $T\in {\mathcal T}$ (case (\ref{item toro en JSJ decomposition})), or else $T$ is in the interior of a Seifert component and it is not $\partial$-parallel. We want to show that in the latter case, we have either that the whole $M$ is Seifert, and $T$ can be put horizontally (case (\ref{item horizontal torus})); or else $T$ can be put vertically (case (\ref{item vertical torus})). Let us then assume we are in the latter case.\par
Now, after Theorem \ref{teorema waldhausen}, there is an isotopy transforming $T$ into a horizontal or vertical torus in the interior of the Seifert component that contains it. Equivalently, there is an isotopy moving the Seifert component and fixing $T$, so that $T$ is either horizontal or vertical in this new Seifert manifold. This produces a new JSJ-decomposition ${\mathcal T}'$, so that $T$ is horizontal or vertical in the Seifert component that contains it. \par
Let us assume that $T$ is horizontal in its Seifert component $S$. Then Lemma \ref{lema horizontal surface} implies that $S$ is a closed manifold. Hence the whole manifold $M$ is Seifert ($M=S$), and we are in case (\ref{item horizontal torus}). Note that in this case ${\mathcal T}'=\emptyset$.\par
If, on the contrary, $T$ is vertical, recall that, by the Enclosing Property (Proposition \ref{proposition enclosing property}), $T$ is not $\partial$-parallel in its Seifert component, that is, $T$ is not isotopic to any component of the boundary of the Seifert component of $M\setminus{\mathcal T}$ containing it. After the isotopy that transforms ${\mathcal T}$ into ${\mathcal T}'$, so that $T$ is a vertical torus of the Seifert component of $M\setminus{\mathcal T}'$ that contains $T$, we will obviously have that $T$ is not $\partial$-parallel in its Seifert component either. Hence we are in case (\ref{item vertical torus}). This proves part (\ref{item JSJ-decomposition}) of Proposition \ref{proposition reduccion}, that is, we have obtained a JSJ-decomposition ${\mathcal T}'$.\par
Now, we want to obtain $f:M\to M$ satisfying items (\ref{item hyperbolic toral automorphism}) and (\ref{item f deja invariante JSJ}) of Proposition \ref{proposition reduccion}. Let us begin by looking for an $f:M\to M$ satisfying item(\ref{item hyperbolic toral automorphism}):
\begin{lemma}\label{lema hyperbolic automorphism}
If $T$ is an Anosov torus of any 3-manifold $M$, that is, with or without boundary, irreducible or not, then there is a diffeomorphism $f:M\to M$ leaving $T$ invariant, such that $f|T$ is a hyperbolic automorphism.
\end{lemma}
\begin{proof}
Let us first consider the case in which $T$ is contained in the interior of $M$. Let $g$ be a diffeomorphism of $M$ leaving $T$ invariant and such that $g|T$ is isotopic to a hyperbolic automorphism $A$. Consider a product neighborhood $T\times[-1,1]$ of $T$. Consider a diffeotopy $h_t:T\to T$ such that $h_0=A\circ g^{-1}$ and $h_1(y)=y$. Define
$$\xi(x)=\left\{\begin{array}
{ll}x& \mbox{if }x\notin T\times[-1,1]\\
h_{|t|}(y)&\mbox{if }x=(y,t)\in T\times [-1,1]
\end{array}\right.$$
Then $f=\xi\circ g$ is the diffeomorphism we are looking for.\par
If $M$ is a manifold with boundary and $T\subset \partial M$, then we can consider a neighborhood of $T$ of the form $T\times [0,1]$. The rest of the proof follows analogously.
\end{proof}
In this way, we have obtained a diffeomorphism $f:M\to M$ satisfying item (\ref{item hyperbolic toral automorphism}) of Proposition \ref{proposition reduccion}. In order to obtain item (\ref{item f deja invariante JSJ}), we shall need the following version of the JSJ-decomposition:
\begin{theorem}[Relative JSJ-decomposition, \cite{jaco_shalen, johannson}] \label{teorema relative jsj} If $M$ is a compact orientable irreducible 3-manifold with incompressible boundary, then there exists a family ${\mathcal T}$ of incompressible annuli and tori, such that $M\setminus {\mathcal T}$ consists of either Seifert or atoroidal and acylindrical components. Any such family ${\mathcal T}$ that is minimal by inclusion is unique up to proper isotopy.
\end{theorem}
Indeed, let $f:M\to M$ be as in Lemma \ref{lema hyperbolic automorphism}.
Now cut $M$ along $T$. The resulting manifold $N$  is in the hypotheses of Theorem \ref{teorema relative jsj}, and $f$
can be extended to this new manifold, since $f(T)=T$. Consider $f({\mathcal T})$, the image by $f$ of the JSJ-decomposition of $M$. Then $f({\mathcal T})$ and ${\mathcal T}$ are JSJ-decompositions for $N$, due to the Enclosing Property (Proposition \ref{proposition enclosing property}). Theorem \ref{teorema relative jsj} implies that there is an isotopy $h_t$ fixing $\partial N$ such that $h_0=id$, and $h_1({\mathcal T})=f({\mathcal T})$.
Now $g=h_1^{-1}\circ f$ is a diffeomorphism satisfying all conditions of Proposition \ref{proposition reduccion} for the JSJ-decomposition ${\mathcal T}$. This finishes the proof.
\begin{obs}\label{remark f deja invariante JSJ} Note that the same idea above shows that, given a closed irreducible orientable 3-manifold $M$, with or without Anosov torus, and given any diffeomorphism $f:M\to M$, a JSJ-decomposition of $M$, ${\mathcal T}$ can be taken so that $f({\mathcal T})={\mathcal T}$. Theorem \ref{teorema jsj} is enough to prove this.\par
This result obviously holds also for compact orientable irreducible 3-manifold with incompressible boundary.
\end{obs}
%
\section{Horizontal Anosov tori}\label{seccion horizontal torus}
We begin the proof of Theorem \ref{theorem principal}. Consider a closed irreducible orientable 3-manifold $M$, and let $T$ be an Anosov torus. Then Proposition \ref{proposition reduccion} states that we need only study three situations:
(\ref{item toro en JSJ decomposition}) $T$ belonging to a JSJ-decomposition, (\ref{item vertical torus}) $T$ being a non-$\partial$-parallel vertical torus in a Seifert component, or (\ref{item horizontal torus}) $T$ being a horizontal torus in a closed Seifert manifold. In this section we study case (\ref{item horizontal torus}). The conclusion is that there are only two manifolds admitting this situation:
\begin{proposition}\label{proposition toro horizontal} Let $M$ be a closed orientable irreducible Seifert manifold that supports a horizontal Anosov torus. Then $M$ is either $\mathbb{T}^3$ or the mapping torus of $-id$ on ${\mathbb{T}^2}$.
\end{proposition}
The rest of this section is devoted to proving this proposition. Let $M$ be a closed orientable irreducible Seifert manifold, and let $T$ be a horizontal torus in $M$. In \cite[Page 30]{hatcher} we can see that only six Seifert manifolds admit horizontal tori:
\begin{enumerate}
\item $M_1=\mathbb{T}^3$,
\item $M_2$ is the mapping torus of $-id$ on ${\mathbb T}^2$, that is, $\mathbb{S}^1\widetilde\times\mathbb{S}^1\widetilde\times\mathbb{S}^1$
\item $M_3$ is the mapping torus of $\left(\begin{array}{cc}-1 & -1\\ 1  &  0\end{array}\right)$
\item $M_4$ is the mapping torus of $\left(\begin{array}{cc}0&-1\\1&0\end{array}\right)$
\item $M_5$ is the mapping torus of $\left(\begin{array}{cc}0&-1\\1&1\end{array}\right)$
\item $M_6=N\cup_{\varphi} N$
\end{enumerate}
In the last case, $N=\mathbb{S}^1\widetilde\times\mathbb{S}^1\widetilde\times [0,1]$ is the twisted $I$-bundle over the Klein bottle and $M_6$ is the closed manifold formed by two copies of $N$, glued together along its boundary. $\partial N$ is a 2-torus and the two copies of $\partial N$ are glued together by the automorphism $\varphi = \left(\begin{array}{cc}0&1\\1&0\end{array}\right)$. $M_6$ is foliated by tori with the exception of two fibers that are Klein bottles.\par
The first five manifolds are torus bundles over $\mathbb{S}^1$ and $M_6$ is a fibration by tori except for two fibers. For all the six manifolds the horizontal torus is isotopic to a fiber. We shall then assume that $T$ is a fiber of the $M_i$. \par
Let $f$ be the hyperbolic automorphism of Proposition \ref{proposition reduccion}. For $M_i$ with $i=1,\dots,5$, the manifolds are mapping tori for some automorphism $h_i$. Since $T$ is a fiber, $h_i$ commutes with $f$.
But only $id$ and $-id$ commute with a hyperbolic automorphism. This implies that $M_3,M_4$ and $M_5$ do not admit horizontal Anosov tori.\newline\par
Let us show that $M_6$ does not admit a horizontal Anosov torus: Indeed, if the manifold is $M_6$, the horizontal Anosov torus $T$ splits the manifold into two components diffeomorphic to $N$. We have that $\partial N=T$. The following general lemma precludes the possibility that $M_6$ admit a horizontal Anosov torus, and finishes the proof of Proposition \ref{proposition toro horizontal}:
\begin{lemma}\label{lema borde toro no es anosov} If $N$ is a compact orientable 3-manifold such that $\partial N$ is a torus $T$, then $T$ is not an Anosov torus.
\end{lemma}
In order to prove Lemma \ref{lema borde toro no es anosov}, we shall need the following:
\begin{lemma}\cite[Lemma 3.5]{hatcher}\label{lema half homology}
 Let $M$ be a compact orientable 3-manifold with boundary $\partial M$. Consider the inclusion $i_*:H_1(\partial M)\hookrightarrow H_1(M)$. Let $\ker(i_*)$ be the kernel of the map induced by the inclusion $i_*$, and let $\rank(H_1(\partial M))$ be the rank of $H_1(\partial M)$. Then
 $$\rank(\ker(i_*))=\frac{1}{2}\rank(H_1(\partial M)).$$
\end{lemma}
Here ``rank" means the number of ${\mathbb Z}$ summands in a direct sum splitting into cyclic groups. If the homology with coefficients in ${\mathbb Q}$ is used, rank can be replaced by ``dimension". \par
In fact, Lemma 3.5 of \cite{hatcher} states that
\begin{equation}\label{formula rango hatcher}
\rank(\image(\partial))=\frac{1}{2}\rank(H_1(\partial M))
\end{equation}
where $\image(\partial)$ stands for the image of the boundary map $\partial: H_2(M,\partial M)\rightarrow H_1(\partial M)$. The fact that the following sequence
$$H_2(M,\partial M)\stackrel{\partial}{\longrightarrow}H_1(\partial M)\stackrel{i_*}{\longrightarrow}H_1(M)$$
is exact implies that $\image(\partial)$ is isomorphic to $\ker(i_*)$, hence $\rank(\image(\partial))=\rank(\ker(i_*))$.\newline \par
To prove Lemma \ref{lema borde toro no es anosov}, let us consider a compact orientable 3-manifold $N$ such that $\partial N$ is a torus. Then $\rank(H_1(\partial N))=2$.
Lemma \ref{lema half homology} implies that the rank of the kernel of the inclusion $i_*:H_1(T)\to H_1(M)$ is one.
This implies that $K=\ker(i_*)$ is a one-dimensional subspace of $H_1(T)$. We shall have that $f_*(K)=K$, where $f_*:H_1(T)\to H_1(T)$ is the isomorphism induced by any diffeomorphism $f:N\to N$. This implies that $f_*$ has an eigenvalue which is $\pm 1$. Hence $f$ cannot be isotopic to a hyperbolic automorphism on $T$. This implies that $T$ cannot be an Anosov torus. This finishes the proof of Lemma \ref{lema borde toro no es anosov}, and hence of Proposition \ref{proposition toro horizontal}.
\section{Vertical Anosov tori}\label{section vertical torus}
Continuing with the proof of Theorem \ref{theorem principal}, we consider an Anosov torus $T$ of an irreducible orientable closed 3-manifold $M$, and study now the situation (\ref{item vertical torus}) of Proposition
\ref{proposition reduccion}; namely, $T$ is a vertical torus in the interior of a Seifert component of the JSJ-decomposition of $M$, that is not $\partial$-parallel. \par
The main result of this section is that in this case, $M$ is like in the case (\ref{item horizontal torus}): either $M$ is ${\mathbb T}^3$ or $M$ is the mapping torus of $-id$ on ${\mathbb T}^2$. More precisely:
\begin{proposition}\label{proposition vertical torus} Let $M$ be a compact connected orientable irreducible Seifert manifold, with or without boundary, admitting a vertical Anosov torus $T$. Then there are only three possibilities:
\begin{enumerate}
\item $M=T\times[-1,1]$,
\item $M={\mathbb T}^3$, or
\item $M$ is the mapping torus of $-id$ on ${\mathbb T}^2$.
\end{enumerate}
\end{proposition}
If $M$ is a Seifert manifold with or without boundary and $T$ is a vertical Anosov torus in $M$, then we can split $M$ by  cutting it along $T$, and we obtain a Seifert manifold with an Anosov boundary torus. Hence we can always consider that $M$ is a Seifert manifold with $T\subset \partial M$. The proof of Proposition \ref{proposition vertical torus} is then reduced to the proof of:
\begin{proposition} \label{proposition solid torus} Let $N$ be a compact connected orientable irreducible Seifert manifold with an Anosov torus $T\subset \partial N$. Then, $N=T\times [0,1]$.
\end{proposition}
Indeed, if $M$ is a manifold as in the hypothesis of Proposition \ref{proposition vertical torus}, and we split $M$ along the vertical Anosov torus $T$, then each component of $M\setminus T$ is a manifold $N$ in the hypotheses of Proposition \ref{proposition solid torus}, and hence each component $N$ is of the form ${\mathbb T}^2\times [0,1]$.
If $M\setminus T$ has two components, this readily implies that $M=T\times[-1,1]$. Otherwise, the manifold $N$ obtained by splitting $M$ along $T$ is connected and is $T\times [0,1]$. Then $M$ is a mapping torus of an automorphism $h$ of $T={\mathbb T}^2$. \par
But, since $T$ is an Anosov torus, there is a diffeomorphism $f:N\rightarrow N$ such that $f|T=f|\partial N$ is a hyperbolic toral automorphism, see Lemma \ref{lema hyperbolic automorphism}. Now, $h$ has to commute with $f$ on $T$. The only possibilities for $h$ are then: $h=id$, $h=-id$ or $h$ is a hyperbolic automorphism of ${\mathbb T}^2$. The last possibility corresponds to a mapping torus that is not a Seifert manifold. Hence, we can only have that $M$ is the mapping torus of $\pm id$ on ${\mathbb T}^2$, as claimed. \newline\par

To finish the proof of Proposition \ref{proposition solid torus}, it is convenient to recall that most Seifert manifolds have a unique Seifert fibration up to isotopy. Namely, we have the following:
\begin{lemma}\cite[Lemma 1.15]{hatcher}\label{lemma non isotopic boundary} The only compact connected {\em orientable} Seifert manifolds with boundary, admitting two Seifert fibrations that are non-isotopic in their boundary are the following:
\begin{enumerate}
\item the solid torus ${\mathbb D}^2\times {\mathbb S}^1$
\item the twisted $I$-bundle over the Klein bottle $\mathbb{S}^1\widetilde\times\mathbb{S}^1\widetilde\times [0,1]$, or
\item the torus cross the interval ${\mathbb T}^2\times [0,1]$
\end{enumerate}
\end{lemma}
Now, let $T\subset \partial N$ be an Anosov torus of $N$ as in Proposition \ref{proposition solid torus}. Then there are two Seifert fibrations of $N$ that are not isotopic on $T$. Indeed, take any Seifert fibration of $N$, and consider the diffeomorphism $f:N\to N$ such that $f|T$ is a hyperbolic automorphism (Lemma \ref{lema hyperbolic automorphism}). The fibration of $N$ restricted to $T$ is not isotopic to its $f$-image (another Seifert fibration) on $T$. Then Lemma \ref{lemma non isotopic boundary} implies that $N$ is either the solid torus, the twisted $I$-bundle over the Klein bundle, or the torus cross the interval. \par
Since Anosov tori are incompressible (Theorem \ref{teorema anosov incompresible}), $N$ is not the solid torus.
If $N$ is the twisted $I$-bundle over the Klein bottle, then $N$ has a connected boundary consisting exactly of one torus, that is, $\partial M={\mathbb T}^2$. Lemma \ref{lema borde toro no es anosov} implies that $T$ cannot be the boundary of $N$. The only possibility left is that $M=T\times [0,1]$. This finishes the proof of Proposition \ref{proposition solid torus}, and hence of Proposition \ref{proposition vertical torus}.
%
\section{Anosov tori in the JSJ-decomposition}\label{section atoroidal}
Let us recall that in Proposition \ref{proposition reduccion}, we reduced the proof of Theorem \ref{theorem principal} to three cases: (\ref{item toro en JSJ decomposition}) $T$ is one of the tori of the JSJ-decomposition, (\ref{item vertical torus}) $T$ is a non-$\partial$-parallel vertical torus in the interior of a compact Seifert manifold, or (\ref{item horizontal torus}) $T$ is a horizontal torus in a closed Seifert manifold. We addressed cases (\ref{item horizontal torus}) and (\ref{item vertical torus}) in Sections \ref{seccion horizontal torus} and \ref{section vertical torus}, respectively. In this section we deal with case (\ref{item toro en JSJ decomposition}). Note that $T$ can be either the boundary of a Seifert component, in which case the component is $T\times [0,1]$, as we proved in Proposition \ref{proposition solid torus}; or else $T$ is the boundary of an atoroidal and acylindrical component. \par
The main result of this section is that an Anosov torus $T$ is never, in fact, a boundary of an atoroidal and acylindrical component of a JSJ-decomposition. This is the most delicate part of the proof of Theorem \ref{theorem principal}. With this result is easy to finish the proof of Theorem \ref{theorem principal}, as it is seen at the end of this section.
\begin{proposition}\label{proposition no atoroidal} Let $M$ be a compact, connected, orientable, irreducible, atoroidal and acylindrical 3-manifold such that $\partial M$ consists of incompressible tori. Then no component of $\partial M$ is an Anosov torus.
\end{proposition}
The strategy is to assume that there is an Anosov torus $T\subset \partial M$, and then use its properties to build an incompressible annulus that is not $\partial$-parallel.\newline\par
\begin{claim} \label{claim surface}For each torus $T\subset \partial M$, there exists a  compact, connected, orientable, incompressible, properly embedded surface $S$, such that $\partial S$ contains an essential curve $\gamma$ in $T$ and $S$ is not a $\partial$-parallel annulus.
\end{claim}
Let us first consider the case $\partial M=T$. Then $\rank (H_1(\partial M))=2$. Equation (\ref{formula rango hatcher}) implies that $\rank (\image(\partial))=1$, hence there exists $\xi\in H_2(M,\partial M)$ such that $\partial \xi\ne 0$ in $H_1(\partial M)$. \cite[Lemma 3.6]{hatcher} states that $\xi$ is represented by an irreducible, properly embedded, compact, orientable surface $S$, possibly non-connected. Since $\partial S$ is non-trivial in $H_1(\partial M)$, there is a connected component of $S$, that is non-trivial in $H_1(\partial M)=H_1(T)$. This component satisfies the claim.\par
In case $T\varsubsetneq \partial M$, we construct a manifold $N$ such that $\partial N=T\sqcup T$: take two copies of $M$, $M\sqcup M$ and glue all corresponding pairs of connected components of boundaries of $M$, except $T\sqcup T$. In this way we obtain a connected compact orientable 3-manifold $N$ such that $\partial N=T\sqcup T$. Note that $H_1(\partial N)\ne 0$, so proceeding as in the previous case, we obtain an irreducible, properly embedded, compact, connected, orientable surface $S$ representing a non-trivial homology class in $H_2(N,\partial N)$ such that $\partial S$ is non-trivial in $H_1(\partial M)$. \par
Without loss of generality, we may assume that $S$ is transverse to $\partial M$. Cutting along $\partial M$ we obtain a new surface $R=S\cap M$, possibly non-irreducible and non-connected, whose boundary is non-trivial in $H_1(\partial M)$, since $R\sqcup R=S$. Let us see that we can cut $R$ along a finite number of curves so that $R$ becomes irreducible.\par
Let $R_i$ be a component of $R$ and suppose $\pi_1(R_i)\hookrightarrow \pi_1(M)$ is not injective. Then there is a disk $D$ realizing this non-injectivity \cite[Corollary 3.3]{hatcher}, that is, there is a disc $D$ such that $\partial D=D\cap R_i$ is a non-null homotopic circle in $R_i$. If we cut $R_i$ along $\partial D$ we obtain a new surface that is in the same homology class, since the cutting curve and $D$ are now duplicated in the boundary of $R_i$, but counted with different signs. Note that this new $R_i$ is also properly embedded.\par
Since $R$ is compact, this surgery simplifies $R_i$. Indeed, either $\partial D$ separates $R_i$, in which case the surgery splits $R_i$ into two components of lower genus; or not, in which case the surgery reduces the genus of $R_i$. We can perform finitely many cuttings until each resulting surface $R_i$ satisfies that $\pi_1(R_i)\hookrightarrow \pi_1(M)$ is injective, so $R_i$ is irreducible.\par
Note that this procedure did not change $\partial R$. So, we have obtained a new $R$ that is irreducible, properly embedded, and such that $\partial R$ is non-trivial in $H_1(\partial M)$, but with a non-trivial element in $H_1(T)$.
Hence, there is a component $R_i$ of $R$ containing an essential curve in $T$, and such that $\partial R_i$ is non-trivial in $H_1(\partial M)$. Then $R_i$ is not a $\partial$-parallel annulus. This finishes Claim \ref{claim surface}. \newline\par

Consider now an Anosov torus $T\subset \partial M$, and let $f:M\to M$ be a diffeomorphism such that $f|T$ is a hyperbolic automorphism. Let $S$ be the surface obtained in Claim \ref{claim surface}, and let $\gamma\subset \partial S$ be an essential curve in $T$, as obtained in Claim \ref{claim surface}.
\begin{claim}\label{claim annulus} For some suitable $n>0$, there exists a properly embedded annulus $A\subset S\cup f^n(S)$ such that one component of $\partial A$ is a closed curve in $T$ formed by one sub-arc $\gamma_1$ of $\gamma$ and one sub-arc $\gamma_2$ of $f^n(\gamma)$.
\end{claim}
Without loss of generality we assume that $S$ is transverse to $f^n(S)$ for all $n>0$. Let $\Gamma_n$ be the set of properly embedded curves of $S\cap f^n(S)$ having an endpoint in $\gamma\cap f^n(\gamma)$. Note that, since $\# f^n(\gamma)\cap \gamma$ goes to infinity as $n\to \infty$, we also have that $\#\Gamma_n\to \infty$ with $n$. \par
On the other hand, the number of non-isotopic properly embedded simple curves contained in $S$ has an upper bound $\kappa$. Note that the same $\kappa$ is also an upper bound for the number of non-isotopic properly embedded simple curves contained in $f^n(S)$, for each $n>0$. Hence, by taking $n$ sufficiently large, we can obtain $\kappa +1$ curves in $\Gamma_n$ that are isotopic in $(S,\partial S)$. At least two of these curves, say $\alpha_1$ and $\alpha_2$, are also isotopic in $(f^n(S),\partial f^n(S))$. \par
The annulus $A$ is built in the following way. Construct a rectangle $R$ by joining $\alpha_1$ and $\alpha_2$ by means of an arc $\gamma_1\subset\gamma$ and and arc $\beta_1\subset\partial S$. Construct another rectangle $R'$ by joining $\alpha_1$ and $\alpha_2$ by means of an arc $\alpha_2\subset f^n(\gamma)$ and an arc $\beta_2\subset \partial f^n(S)$. This is possible since $\alpha_1$ and $\alpha_2$ belong to $\Gamma_n$ and are isotopic both in $(S,\partial S)$ and $(f^n(S),\partial f^n(\partial S))$. \par
In this way we obtain a properly embedded annulus $A$ bounded by $\gamma_1\cup\gamma_2\subset \gamma\cup f^n(\gamma)\subset T$ and $\beta_1\cup\beta_2\subset\partial M$. This proves Claim \ref{claim annulus}.\newline\par
The rest of the section is devoted to proving that the annulus $A$ obtained in Claim \ref{claim annulus} is non-$\partial$-parallel.
\begin{claim}\label{claim a non parallel} The anuulus $A$ is non-$\partial$-parallel.
\end{claim}
Let us begin by proving that $A$ is incompressible. Indeed, we shall see that $\gamma_1\cup \gamma_2\subset A\cap T$ is an essential curve in $T$. We loose no generality in assuming that $\gamma$ is a line in $T$ with rational slope. Since $f|T$ is a hyperbolic automorphism, $f^n(\gamma)$ is a line with different slope. The lifts of $\gamma$ and $f^n(\gamma)$ to the universal covering cannot enclose a region, as it is seen in the Figure \ref{closed_curve}. Thus any closed curve formed by a segment in $\gamma$ and a segment in $f^n(\gamma)$, like $\gamma_1\cup \gamma_2$, is essential in $T$. This implies that $A$ is incompressible.\par
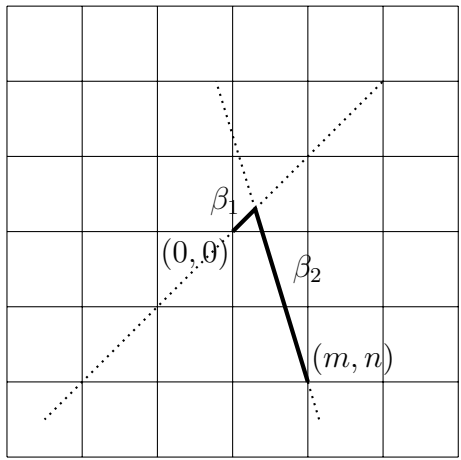
\begin{figure}[h]
\begin{tikzpicture}
 \begin{scope}[ultra thick]
 \draw node at (2.9,.4) {$\beta_1$} (3,0) -- (3.3,.3) node at (4,-.5) {$\beta_2$} -- (4,-2);
 \end{scope}
 \draw (0,-3) grid +(6,6);
 \node at (2.5,-.3) {$(0,0)$};
 \node at (4.6,-1.7){$(m,n)$};
 \draw[dotted,thick] (3.3,.3) -- (5,2);
 \draw[dotted,thick] (3.3,.3) -- (.5,-2.5);
 \draw[dotted,thick] (3.3,.3) -- (2.78,2);
 \draw[dotted,thick] (4,-2) -- (4.15,-2.5);
\end{tikzpicture}
\caption{\label{closed_curve}
 A closed curve in $T$ formed by segments of different slope}
\end{figure}

In order to prove that $A$ is not $\partial$-parallel, let us introduce the following concept:
\begin{definition}\cite[Page 14]{hatcher}\label{definicion d-incompresible} Let $S$ be a surface properly embedded in $M$, that is, $\partial S\subset \partial M$. A {\em $\partial$-compressing disk} $D\subset M$ is a disk such that $\partial D$ consists of two arcs $\alpha$ and $\beta$ such that $\alpha\cap \beta=\partial \alpha=\partial\beta$, where $\alpha=D\cap S$ and $\beta=D\cap \partial M$, see Figure \ref{figura compressing disk}. \par
A properly embedded surface $S$ is {\em $\partial$-incompressible} if for each compressing disk $D$ there is a disk $D'\subset S$ with $\alpha \subset \partial D'$ and $\partial D'\setminus \alpha\subset \partial S$, see Figure \ref{figura d incompresible}.
\end{definition}
\begin{figure}\begin{tikzpicture}
 \draw[thick] (-3,1) -- (-1,4) -- (6,4) -- (4,1) -- cycle;
 \draw[thick] (-1,1.5) -- (0.5,3.5) -- (3,3.5) -- (1.5,1.5) -- cycle;
 \draw[thick] (-1,1.5) parabola[parabola height=2cm] (0.5,3.5);
 \draw[thick] (1.5,1.5) parabola[parabola height=2cm] (3,3.5);
 \draw[thick] (-.3,4.5) -- (2.2,4.5);%
 \draw[thick,fill=transparent!5,opacity=.5] (.25,1.5) parabola[parabola height=2cm] (1.75,3.5) -- cycle;
  \node at (6,2.5) {$\partial M$};
  \node at (1,3.3) {$D$};
   \node at (1,5) {$S$};
\end{tikzpicture}
\caption{\label{figura compressing disk} A $partial$-compressing disk $D$}
\end{figure}
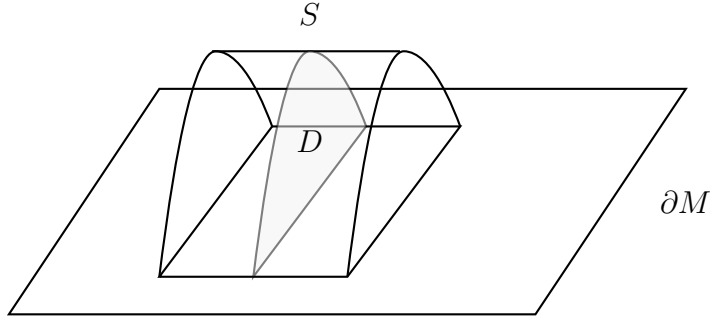
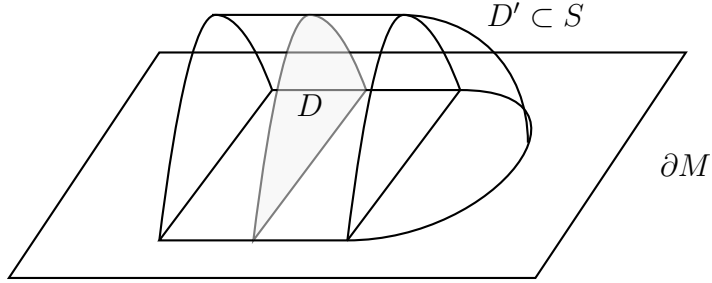
\begin{figure}\begin{tikzpicture}
 \draw[thick] (-3,-4) -- (-1,-1) -- (6,-1) -- (4,-4) -- cycle;
 \draw[thick] (-1,-3.5) -- (0.5,-1.5) -- (3,-1.5) -- (1.5,-3.5) -- cycle;

 \draw[thick] (3,-1.5) .. controls +(0:2cm) and +(0:2cm) .. (1.5,-3.5);
 \draw[thick] (2.2,-.5) .. controls +(0:1cm) and +(90:1cm) .. (3.9,-2.2);
 \draw[thick] (-1,-3.5) parabola[parabola height=2cm] (0.5,-1.5);
 \draw[thick] (1.5,-3.5) parabola[parabola height=2cm] (3,-1.5);
 \draw[thick] (-.3,-.5) -- (2.2,-.5);
 \draw[thick,fill=transparent!5,opacity=.5] (.25,-3.5) parabola[parabola height=2cm] (1.75,-1.5) -- cycle;
 \node at (6,-2.5) {$\partial M$};
 \node at (4,-.5) {$D'\subset S$};
  \node at (1,-1.7) {$D$};
\end{tikzpicture}
\caption{\label{figura d incompresible} A $\partial$-incompressible surface $S$}
\end{figure}
Let us also recall the following property of incompressible surfaces;
\begin{lemma}\cite[Lemma 1.10]{hatcher}\label{lema d incompresible o d paralelo} Let $N$ be a compact irreducible 3-manifold, such that $\partial N$ consists of incompressible tori. If $S$ is a connected incompressible surface properly embedded in $N$, then either $S$ is a $\partial$-parallel annulus or else $S$ is $\partial$-incompressible.
\end{lemma}
So, it remains to prove that $A$ is not $\partial$-parallel. Arguing by contradiction, suppose that $A$ is $\partial$-parallel. This implies that $\partial A\subset T$, and also, that there exists a $\partial$-compressing disk $D$ with $\alpha\subset \partial D$ for all arcs $\alpha$ properly embedded in $A$ and with endpoints in different components of $\partial A$ (see Figure \ref{figure2}). In particular, we can choose an arc $\alpha=\alpha_1$ as in Claim \ref{claim annulus}, that is, an arc $\alpha_1\subset S\cap f^n(S)$, with one endpoint in $\gamma\cap f^n(\gamma)$ and the other endpoint in the other component of $\partial A$. The rest of the boundary of $D$, $\partial D\setminus \alpha_1$ is contained in $T$. \par

\begin{figure}[b]
\begin{tikzpicture}

\draw[fill=gray,fill opacity=.5,ultra thick] (2.87,-.25) .. controls +(20:1cm) and +(-20:1cm) .. (2.87,-3.25) --  (2.87,-.25);
\draw (3,0) .. controls +(20:1cm) and +(-20:1cm) .. (3,-3);
\draw (1,0) .. controls +(10:-1cm) and +(-10:-1cm) .. (1,-3);
\draw (2,0) ellipse (1cm and .5cm);
\draw (2,-3) ellipse (1cm and .5cm);
\draw[dotted] (3.5,-0,4) ellipse (1.9cm and .7cm);
\draw[dotted] (3.5,-1,4) ellipse (1.5cm and .5cm);
\draw[dotted] (2,-.4) ellipse (1.5cm and .5cm);
\draw (1,0) -- (1,-3);
\draw (3,0) -- (3,-3);
\node at (1.5,-1.5) {$T$};
\node at (3.2,-1.5) {$D$};
\node at (-.5,-1.5) {$A$};
\node at (4.2,-.5) {$\alpha_1$};
\draw[->] (4,-.5) -- (3.4,-.5);
\end{tikzpicture}
\caption{\label{figure2} $\partial$-parallel annulus}
\end{figure}
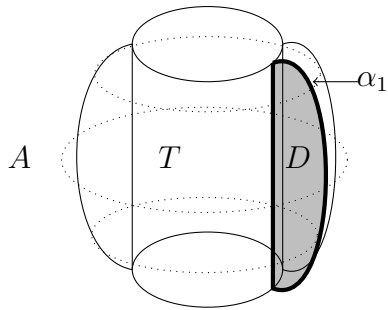

Now, Claim \ref{claim surface} and Lemma \ref{lema d incompresible o d paralelo} imply that both $S$ and $f^n(S)$ are $\partial$-incompressible. Hence, since $\alpha_1$ is properly embedded in $S$ and in $f^n(S)$, we have that $D$ is a $\partial$-compressing disk for $S$ and for $f^n(S)$.\par
$\partial$-incompressibility of $S$ and $f^n(S)$ implies that there are two disks $D_1\subset S$ and $D_2\subset f^n(S)$ such that $\partial D_1=\alpha_1\cup l_1$, where $l_1$ is a segment in $\gamma$, and  $\partial D_2=\alpha_1\cup l_2$, where $l_2$ is a segment in $f^n(\gamma)$. \par
The set $D'=D_1\cup D_2$ is an immersed disk, and $\partial D'=l_1\cup l_2$. Now, $l_1\subset\gamma$ and $l_2\subset f^n(\gamma)$. As we have said at the beginning of this proof, this implies that $\partial D'$ is essential in $T$, so $\partial D'$ cannot be a disk. This implies that $A$ is non-$\partial$-parallel.
and finishes Claim 3.\newline\par
In conclusion, we have seen that the existence of an Anosov tori in the boundary of a 3-manifold whose boundary consists of incompressible tori, implies that the 3-manifold is not acylindrical. This proves Proposition \ref{proposition no atoroidal}.

\section{Proof of Theorems \ref{theorem principal} and \ref{theorem principal con borde}}\label{s_teo1}

In this Section we finish Theorem \ref{theorem principal}, and prove Theorem \ref{theorem principal con borde}. Let us begin by finishing the necessity part of Theorem \ref{theorem principal}.\par
Recall that Proposition \ref{proposition reduccion} reduced the proof of the necessity part of Theorem \ref{theorem principal} to three cases: case (\ref{item toro en JSJ decomposition}): when the Anosov torus is part of the cutting tori of the JSJ-decomposition, case (\ref{item vertical torus}): when the Anosov torus is in the interior of a Seifert component of the JSJ-decomposition, and is not $\partial$-parallel, and (\ref{item horizontal torus}): when $M$ is a Seifert manifold and the Anosov torus is horizontal.\par
In Proposition \ref{proposition toro horizontal}, we show that in case (\ref{item horizontal torus}), then $M$ is either the 3-torus, or else the mapping torus of $-id$. This proves Theorem \ref{theorem principal} in this case. \par
In Proposition \ref{proposition vertical torus}, we prove that if $M$ is a compact connected irreducible Seifert manifold with or without boundary, admitting an Anosov torus $T$, then we have that either $M$ is as in case (\ref{item horizontal torus}), namely, $M$ is the 3-torus, or the mapping torus of $-id$; or else, $M$ is $T\times[0,1]$. But in this last case, $T$ is $\partial$-parallel. So, this proves Theorem \ref{theorem principal} in case (\ref{item vertical torus}). \par
The last case left is (\ref{item toro en JSJ decomposition}), when the Anosov torus is one of the cutting tori of a minimal JSJ-decomposition. Proposition \ref{proposition no atoroidal} shows that the Anosov torus cannot bound an atoroidal an acylindrical component of the JSJ-decomposition. Hence, the Anosov torus is a component of the boundary of a compact connected irreducible Seifert manifold. Proposition \ref{proposition vertical torus} shows that the Seifert component having the Anosov torus $T$ as part of its boundary is $T$ cross the interval. Hence the boundary of the Seifert component of the Anosov torus consists of two isotopic tori. Since the cutting tori of the JSJ-decomposition were taken to be a minimal family, this implies that the JSJ-decomposition consists of only one torus, the Anosov torus $T$. If we cut $M$ along $T$ we obtain $T$ cross the interval. Let $A=f|_T$ be the hyperbolic automorphism obtained since $T$ is an Anosov torus, and let $g:T\to T$ be a gluing diffeomorphism so that when we identify $x$ with $g(x)$ we reobtain $M$. Then $g$ commutes with $A$. This gives us three classes of diffeomorphisms: isotopic to $\pm id$, which would give us a Seifert $M$, or isotopic to a hyperbolic automorphism, which would give us that $M$ is the mapping torus of a hyperbolic automorphism. In the first situations we would have that there are no cutting tori in the JSJ-decomposition, so we are in the last situation, and this finishes the proof of the necessity part of Theorem \ref{theorem principal}.\par
The sufficiency part, as we have said is straightforward. If $M$ is the 3-torus, we can take $f=\left(
                                                                                                  \begin{array}{cc}
                                                                                                    2 & 1 \\
                                                                                                    1 & 1 \\
                                                                                                  \end{array}
                                                                                                \right)\times id$ on
${\mathbb T}^2\times {\mathbb S}^1$, this gives us infinitely many Anosov tori. If $M$ is the mapping torus of a hyperbolic automorphism $A$ over ${\mathbb T}^2$, there is a natural flow $f_t$ which is the suspension of $A$. The time-one map of this flow $f$ leaves invariant infinitely many tori, on which its dynamics is hyperbolic. Finally, let $M$ be the mapping torus of $-id$. Cut $M$ along an incompressible torus. We obtain a torus cross the interval. Define $g=\left(                                                                                                  \begin{array}{cc}
2 & 1 \\
1 & 1 \\
\end{array}
\right)\times id$ on this new manifold with boundary $\tilde M={\mathbb T}^2\times[0,1]$. To reobtain $M$ we identify ${\mathbb T}^2\times\{0\}$ with ${\mathbb T}^2\times\{1\}$ by means of the map $(x,0)\mapsto(-x,1)$. But this map commutes with $g$ on the boundary of $\tilde M$, hence $g$ extends to a diffeomorphism on $M$ leaving invariant infinitely many Anosov tori. This finishes the proof of Theorem \ref{theorem principal}.\newline\par
Finally, let us prove Theorem \ref{theorem principal con borde}. Let $M$ be a compact orientable irreducible 3-manifold with non-empty boundary consisting of incompressible tori, and admitting an Anosov torus $T$. Duplicate $M$ to obtain $M\sqcup M$ and glue along the corresponding boundary components. In this way we obtain a closed orientable irreducible 3-manifold $N$. Now, $N$ is in the hypotheses of Theorem \ref{theorem principal}, so if we cut $N$ along $T$ we obtain $T$ cross the interval. All the incompressible tori in the boundary of $M$ are hence isotopic to $T$, this implies that $M$ is in fact $T$ cross the interval. This proves Theorem \ref{theorem principal con borde}.

\end{document}